\documentclass[11pt,bezier]{article}
\usepackage{amsmath,graphicx,amssymb,amsfonts,appendix}

\newtheorem{prethm}{{\bf Theorem}}

\newenvironment{thm}{\begin{prethm}{\hspace{-0.5
               em}{\bf}}}{\end{prethm}}
\newtheorem{prepro}{{\bf Theorem}}

\newtheorem{preprop}{{\bf Proposition}}

\newtheorem{preex}{{\bf Example}}

\newenvironment{ex}{\begin{preex}{\hspace{-0.5
               em}{\bf}}}{\end{preex}}

\newtheorem{precor}{{\bf Corollary}}

\newenvironment{cor}{\begin{precor}{\hspace{-0.5
               em}{\bf}}}{\end{precor}}
\newtheorem{preconj}{{\bf Conjecture}}

\newtheorem{preremark}{{\bf Remark}}

\newtheorem{prelem}{{\bf Lemma}}

\newenvironment{lem}{\begin{prelem}{\hspace{-0.5
               em}{\bf}}}{\end{prelem}}
\newtheorem{preproof}{{\bf Proof.}}

\newenvironment{proof}[1]{\begin{preproof}{\rm
               #1}\hfill{$\Box$}}{\end{preproof}}

\title{\bf Cospectral lifts of graphs}
\author{\large
   F. Ramezani$^{1}$
    \\{\it Department of Mathematics, Faculty of Science,}
\\ {\it  K.N.Toosi University of Technology, Tehran, Iran,}
\\ {\it P.O. Box 16315-1618}1
 \\[.3cm]
 }
\date{}
\begin{document}
\maketitle  \footnotetext[1]{\tt Corresponding Author, Email
address: ramezani@kntu.ac.ir}

\begin{abstract}\rm\noindent
We prove that for a pair of cospectral graphs $G$ and $H$, there
exist their non trivial lifts $G'$ and $H'$ which are cospectral.
More over for a pair of cospectral graphs on $6$ vertices, we
find some cospectral lifts of them.

\vspace{1cm}
\end{abstract}

\section{Introduction}
Let $G=(V,E)$ be a simple graph on the vertex set
$V(G)=\{v_1,v_2,\ldots,v_n\}$ and edge set $E$. The {\it adjacency
matrix} of $G$ is an $n$ by $n$ matrix $A(G)$ whose $(i,j)$-th
entry is $1$ if vertices $v_i$ and $v_j$ are adjacent and $0$,
otherwise. The \textit{spectrum} of $G$ is the multi-set of
eigenvalues of $A(G)$. Two graphs $G$ and $G'$ are called {\it
cospectral} if they share the same spectrum. We say $G$ is
\textit{determined by spectrum} (\textit{DS} for short) if it has
no non-isomorphic cospectral mate.

The problem of constructing cospectral graphs, has been
investigated by some authors. For a survey of results on this
area we refer the reader to \cite{DH,DH1,RA}. In \cite{GHRATA} the
authors have used the concept $m$-cospectrality to construct new
cospectral graphs. Haemers et all in \cite{HABR} have considered
Godsil-McKay switching method to construct non-isomorphic
cospectral graphs see the paper for more details. In this article
we use the concept lift of graphs to construct new non-isomorphic
cospectral graphs from given small cospectral pairs of graphs.

\section{Preliminaries}

In this section we mention some basic definitions and results
which will be used during the paper. We denote by $\dot{E}$ the
set of all ordered pairs $\{(i,j)|\textrm{ } i< j , \textrm{ }
\{v_i,v_j\}\in E \}$. For an Abelian group of order $k$, say $Gr$,
a $k-$Abelian signature $s$ of the graph $G$ is a map
$s:\dot{E}\longrightarrow Gr$. A $k$-Abelian lift of the graph
$G$, associated with the signature $s$, which is denoted by
$G(s)$, is a graph on the vertex set $V(G)\times [k]$
($[k]=\{0,1,\ldots,k-1\}$, $Gr=(\{g_0,g_1,\ldots,g_{k-1}\},*$)),
where for any $(i,j)\in \dot{E}$ and $a,b\in [k]$ there is an edge
between $(v_i,a)$ and $(v_j,b)$ if and only if $s(i,j)*g_a=g_b$.
Note that in the graph $G(s)$, for any $(i,j)\in \dot{E}$, there
is a matching between the vertex sets $V_i=\{v_i\}\times [k]$ and
$V_j=\{v_j\}\times [k]$. If a graph have $m$ edges there may be
$k^m$ different $k$-Abelian lifts of $G$, since the sets $V_i,V_j$
are matched in $k$ different ways. If the signature $s$ maps all
pairs to the same element $g\in Gr$, then we denote the
corresponding lift $G(s)$ with $G_g$. We illustrate the
definition of the $k$-lifts of a graph in the following figure.
In the following graph the graph $G$ is the cycle $C_4,$ and the
corresponding signature is $s:\dot{E}\longrightarrow
\mathbb{Z}_2$, with $s(1,3)=0,s(1,4)=0,s(2,3)=1,s(2,4)=0.$

\vspace{1cm}

${\hspace{1cm}\put(100,0){\line(0,1){50}}\put(100,0){\line(1,1){50}}\put(150,0){\vdots
 }\put(150,12){\vdots}\put(150,24){\vdots}\put(150,36){\vdots}\put(150,48){.}
 \put(150,0){\line(-1,1){50}}\put(100,0){\circle*{5}}\put(92,-5){4}\put(93,53){1}\put(153,-5){3}\put(153,53){2}\put(150,50){\circle*{5}}
 \put(100,50){\circle*{5}}\put(150,0){\circle*{5}}\put(170,25){$\longrightarrow$}
 \put(210,0){\circle*{4}}\put(220,0){\circle*{4}}\put(210,50){\circle*{4}}\put(220,50){\circle*{4}}
 \put(270,0){\circle*{4}}\put(260,0){\circle*{4}}\put(270,50){\circle*{4}}\put(260,50){\circle*{4}}
 \put(210,0){\line(0,1){50}}\put(220,0){\line(0,1){50}}\put(210,0){\line(1,1){50}}\put(220,0){\line(1,1){50}}
 \put(270,0){\line(-1,1){50}}\put(260,0){\line(-1,1){50}}\put(270,0){\line(-1,5){10}}\put(260,0){\line(1,5){10}}}\put(150,50){\circle*{5}}
 \put(100,50){\circle*{5}}\put(150,0){\circle*{5}}$
$$\hspace{-1.5cm}G\hspace{3.5cm} G(s)$$

 $$\textrm{\textbf{Figure1.} 2-\textrm{lift of } G \textrm{ corresponding to the signature }s}$$

Let $Gr=(\{g_1=1,g_2,\ldots,g_{n}\},*)$ be a group of order $n$.
For any group element say $g\in Gr$ there is an $n\times n$
permutation matrix $P_g$ in correspondence, which is defined
bellow,$$P_g(i,j)=
  \begin{cases}
    1 & \text{ \textrm{if} } g_i*g=g_j, \\
    0 & \text{otherwise}.
  \end{cases}
$$

\begin{lem} {\rm The function $\phi:Gr\rightarrow SL(n,\mathbb{R})$,
where $SL(n,\mathbb{R})$ is the set of $n\times n$ real
non-singular matrices and $\phi(g)=P_g$, is a group homomorphism.}
\end{lem}

The eigenvalues of the graph $G(s)$ has been studied in the
literature. For instance in the following theorem from
\cite{MOTA} the authors have obtained the eigenvalues of Abelian
$t$-lifts. See \cite{MOTA} for more details and the notations.

\begin{thm}\label{mota} { \rm Let $G$ be a multigraph and $\phi$ be a
signature assignment to an Abelian group. Let $\beta$ be a common
basis of eigenvectors of the permutation matrices in the image of
$\phi$. For every $\mathbf{x}\in \beta$, let $A_\mathbf{x}$ be
the matrix obtained from the adjacency matrix of $G$ by replacing
any $(u,v)$-entry of $A(G)$ by $\sum_{(e,u,v)\in
\overrightarrow{E}(G)}\lambda_\mathbf{x}(\phi(e,u,v))$. Then the
spectrum of the $t$-lift $G(\phi)$ of $G$ is the multiset union
of the spectra of the matrices $A_\mathbf{x} (\mathbf{x} \in
\beta)$.}
\end{thm}

\section{Main result}
Our main problem here is "for given pair of cospectral graphs $G$
and $H$, is there $k$-Abelian signatures $s,s'$ which $G(s)$ and
$H(s')$ are cospectral?". We look for general answers of this
question.

It is known that for $l,l'\in Gr$ the permutation matrices
$P_l,P_{l'}$ commute, so they have common basis of eigenvectors.
The following theorem is a straight consequence of Theorem
\ref{mota}.

\begin{thm} \label{ma} {\rm Let G be a graph and $s$ be a $k$-cyclic signature of $G$. Let
$\beta$ be a common basis of eigenvectors of the permutation
matrices in the image of $s$. For every $\mathbf{x} \in \beta$,
let $A_\mathbf{x}$ be the matrix defined bellow
$$A_\mathbf{x}(i,j)=
  \begin{cases}
   \lambda_\mathbf{x}(P_{s(i,j)}) &  i<j, \\
   0& i=j,\\
   \lambda_\mathbf{x}^{-1}(P_{s(i,j)}) & i>j.
  \end{cases}
$$ Then
the spectrum of $G(s)$ is the multi-set union of the spectra of
the matrices $A_\mathbf{x}(\mathbf{x}\in\beta)$.}
\end{thm}

\begin{lem} {\rm Let $Gr$ be a group of order $n$. For any $g\in Gr$, any eigenvalues of the permutation
matrix $P_g$ is an $n$'th root of unity.}
\end{lem}
\begin{proof} {The assertion follows by the fact that the order of any element
in the group divides the order of group. Hence $g^n=1$, and
therefore $P_g^n=I_n.$ Hence the minimal polynomial of $P_g$, say
$m(P_g,x)$ divides the polynomial $x^n-1,$ thus the assertion
follows. }
\end{proof}

\begin{lem} {\rm If $G$ and $H$ are cospectral graphs on $n$ verices
and $Gr$ be a finite group of order $t$. If for $g\in Gr$ the
matrix $P_g$ is symmetric, then the graphs $G_g$ and $H_g$ are
cospectral. }
\end{lem}
\begin{proof} {Since the signature corresponds the fixed element $g$ to all the edges
of the graph $G$, and $P_g^{-1}=P_g,$ then by Theorem \ref{mota},
the eigenvalues of the graph $G_g$ are the multi-set union of the
matrices $\omega_i A(G)$, where $\omega_i$'s are the eigenvalues
of $P_g$ for $i=1,2,\ldots,n.$ On the other hand the eigenvalues
of $\omega_i A(G)$ are $\omega_i \lambda_j$ where $\lambda_j$ is
the $j$'th eigenvalue of $G$. Hence the spectrum of $G_g$ and
$H_g$ are the multi-set $\{\omega_i
\lambda_j\}_{i=1,\ldots,t}^{j=1,\ldots,n}.$

 }
\end{proof}

%The above Lemma states that if the images of both signatures
%$s,s'$ equals the set $\{l\}$ then the graphs $G(s)$ and $H(s')$
%are cospectral.

\subsection{Examples} We consider two cospectral graphs $G$ and
$H$, shown in the Figure 1. We try to find possible Abelian lifts
of them say $G(s)$ and $H(s')$ which are also cospectral.

  \includegraphics[width=18cm]{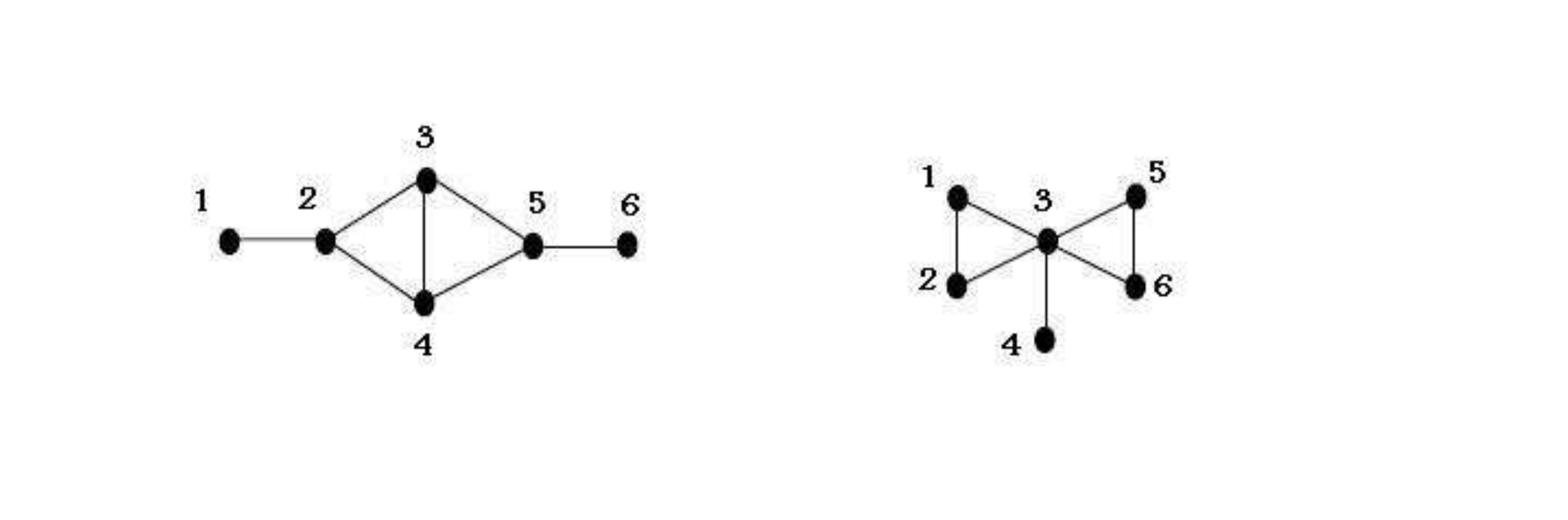}\\
  \vspace{-1cm}
  $$ \textrm{\textbf{Figure1.} Two  cospectral  graphs } G \textrm{ and } H$$\label{as}

We first consider all possible $t$-Abelian signatures of the
graphs $G$ and $H$, suppose that the matrices
$A(G)_\mathbf{x},A(H)_\mathbf{x}$, corresponding to the graphs
$G$, $H$ and their prescribed signatures, which is introduced in
Theorem \ref{mota} are of the following general forms,

$$A(G)_{\mathbf{x}}=\begin{pmatrix}
  0 & u & 0 & 0 & 0 & 0 \\
  u^{-1} & 0 & v & w & 0 & 0 \\
  0 & v^{-1} & 0 & x & y & 0 \\
  0 & w^{-1} & x^{-1} & 0 & z & 0 \\
  0 & 0 & y^{-1} & z^{-1} & 0 & r \\
  0 & 0 & 0 & 0 & r^{-1} & 0
\end{pmatrix}\textrm{, }A(H)_\mathbf{x}=\begin{pmatrix}
  0 & u_1 & v_{1} & 0 & 0 & 0 \\
  u^{-1}_{1} & 0 & w_{1} & 0 &0 & 0 \\
  v^{-1}_{1} & w^{-1}_{1} & 0 & x_{1} & y_{1} & z_{1} \\
  0 &0 & x^{-1}_{1} & 0 & 0 & 0 \\
  0 & 0 & y^{-1}_{1} & 0 & 0 & r_1 \\
  0 & 0 & z^{-1}_{1} & 0 & r^{-1}_{1} & 0
\end{pmatrix}.$$

note that $r,u,v,\ldots,z,r_1,u_1,v_1,\ldots,z_1$ are the complex
variables and stand for the eigenvalues of the permutation
matrices corresponding to each edge. Using Theorem \ref{ma} we
find sufficient conditions on the signatures such that the
corresponding lifts become cospectral.

\begin{thm} \label{as}{ Let $s,s'$ be $k$-Abelian lifts on the graphs $G,H$
respectively. If the following situations hold, the graphs $G(s)$
and $H(s')$ are cospectral.

{\large \begin{itemize}
  \item $\frac{w}{v}=\frac{y}{z}$
  \item $2(\frac{xv}{w}+\frac{w}{xv})=\frac{y_1r_1}
{z_1}+\frac{z_1}{y_1r_1}+\frac{u_1w_1}{v_1}+\frac{v_1}{u_1w_1}.$

\end{itemize}}
}
\end{thm}

\begin{proof} {We consider all posibilities for the matrices $A(G)_\mathbf{x},A(H)_\mathbf{x}$
corresponding to the graphs $G,H$. Comparing the coefficients of
$\chi(A(G)_\mathbf{x},t),\chi(A(H)_\mathbf{x},t)$, the equality
holds if and only if
$$2=\frac{wz}{vy}+\frac{vy}{wz},\hspace{2cm}(1)$$
$$\frac{xz}{y}+\frac{y}{xz}+\frac{xv}{w}+\frac{w}{xv}=\frac{y_1r_1}
{z_1}+\frac{z_1}{y_1r_1}+\frac{u_1w_1}{v_1}+\frac{v_1}{u_1w_1}.\hspace{2cm}(2)$$
The first equation follows by comparing the coefficients of $t^2$
and the second one follows by comparing the coefficients of
$t,t^3$. Consider the first equality above, note the variables
are the $n$'th roots of unity hence the equality $(1)$ holds if
and only if $wz=vy$, hence the first assertion of the statement
follows. The second assertion follows by the Equation (2) and the
first equality. Hence the graphs according to the mentioned
signatures are cospectral.}
\end{proof}
\begin{cor} {The following constraints on the variables $r,u,v,\ldots,z,r_1,u_1,v_1,\ldots,z_1$
will give cospectral lifts of the graphs $G$ and $H$.
$$z=\frac{vy}{w},\textrm{ }u_1=y_1=x,\textrm{ }w_1=r_1=v,\textrm{ }z_1=w$$ }
\end{cor}
\begin{ex} {\rm In the graphs $G$ and $H$ the following signatures have the condition stated in
Theorem \ref{as}. Hence the graphs $G(s)$ and $H(s')$ are
cospectral. The group members are denoted in the cyclic
representation.
$$s(1,2)=(1,2,3),s(2,3)=(1,3,2),s(2,4)=\textrm{id},$$$$ s(3,4)=(1,3,2),s(3,5)=(1,3,2),s(4,5)=(1,2,3),s(5,6)=(1,2)$$
$$s'(1,2)=(1,3,2),s'(1,3)=\textrm{id},s'(2,3)=(1,3,2),$$$$s'(3,4)=\textrm{id},s'(3,5)=(1,3,2),s'(3,6)=\textrm{id},s'(5,6)=(1,3,2)$$
The adjacency matrices of the graphs $G(s)$ and $H(s')$ are of
the following forms.
$$A(G(s)) = \begin{pmatrix}0, 0, 0, 0, 1, 0, 0, 0, 0, 0, 0,
0, 0, 0, 0, 0, 0, 0\\ 0, 0, 0, 0, 0, 1, 0, 0, 0, 0, 0, 0, 0, 0,
0, 0, 0, 0\\ 0, 0, 0, 1, 0, 0, 0, 0, 0, 0, 0, 0, 0, 0, 0, 0, 0,
0\\ 0, 0, 1, 0, 0, 0, 0, 0, 1, 1, 0, 0, 0, 0, 0, 0,
0, 0\\ 1, 0, 0, 0, 0, 0, 1, 0, 0, 0, 1, 0, 0, 0, 0, 0, 0, 0\\
0, 1, 0, 0, 0, 0, 0, 1, 0, 0, 0, 1, 0, 0, 0, 0, 0, 0\\ 0, 0, 0,
0, 1, 0, 0, 0, 0, 0, 0, 1, 0, 0, 1, 0, 0, 0\\ 0, 0, 0, 0, 0, 1,
0, 0, 0, 1, 0, 0, 1, 0, 0, 0, 0, 0\\ 0, 0, 0, 1, 0, 0, 0, 0, 0,
0, 1, 0, 0, 1, 0, 0, 0, 0\\ 0, 0, 0, 1, 0, 0, 0, 1, 0, 0, 0, 0,
0, 1, 0, 0, 0, 0\\ 0, 0, 0, 0, 1, 0, 0, 0, 1, 0, 0, 0, 0, 0, 1,
0, 0, 0\\ 0, 0, 0, 0, 0, 1, 1, 0, 0, 0, 0, 0, 1, 0, 0, 0, 0,
0\\
0, 0, 0, 0, 0, 0, 0, 1, 0, 0, 0, 1, 0, 0, 0, 0, 1, 0\\ 0, 0, 0,
0, 0, 0, 0, 0, 1, 1, 0, 0, 0, 0, 0, 1, 0, 0\\0, 0, 0, 0, 0, 0, 1,
0, 0, 0, 1, 0, 0, 0, 0, 0, 0, 1\\ 0, 0, 0, 0, 0, 0, 0, 0, 0, 0,
0, 0, 0, 1, 0, 0, 0, 0\\ 0, 0, 0, 0, 0, 0, 0, 0, 0, 0, 0, 0, 1,
0, 0, 0, 0, 0\\ 0, 0, 0, 0, 0, 0, 0, 0, 0, 0, 0, 0, 0, 0, 1, 0,
0, 0\end{pmatrix},A(H(s'))=\begin{pmatrix}0, 0, 0, 0, 0, 1, 0, 0, 1, 0, 0, 0, 0, 0, 0, 0, 0, 0\\ 0, 0, 0, 1, 0, 0, 1, 0, 0, 0, 0, 0, 0, 0, 0, 0, 0, 0\\
 0, 0, 0, 0, 1, 0, 0, 1, 0, 0, 0, 0, 0, 0, 0, 0, 0, 0\\  0, 1, 0, 0, 0, 0, 0, 1, 0, 0, 0, 0, 0, 0, 0, 0, 0, 0\\
   0, 0, 1, 0, 0, 0, 0, 0, 1, 0, 0, 0, 0, 0, 0, 0, 0, 0\\    1, 0, 0, 0, 0, 0, 1, 0, 0, 0, 0, 0, 0, 0, 0, 0, 0, 0\\
     0, 1, 0, 0, 0, 1, 0, 0, 0, 0, 0, 1, 0, 0, 0, 0, 1, 1\\     0, 0, 1, 1, 0, 0, 0, 0, 0, 0, 1, 0, 0, 1, 1, 0, 0, 0\\
       1, 0, 0, 0, 1, 0, 0, 0, 0, 1, 0, 0, 1, 0, 0, 1, 0, 0\\        0, 0, 0, 0, 0, 0, 0, 0, 1, 0, 0, 0, 0, 0, 0, 0, 0, 0\\
         0, 0, 0, 0, 0, 0, 0, 1, 0, 0, 0, 0, 0, 0, 0, 0, 0, 0\\          0, 0, 0, 0, 0, 0, 1, 0, 0, 0, 0, 0, 0, 0, 0, 0, 0, 0\\
           0, 0, 0, 0, 0, 0, 0, 0, 1, 0, 0, 0, 0, 0, 0, 0, 0, 1\\            0, 0, 0, 0, 0, 0, 0, 1, 0, 0, 0, 0, 0, 0, 0, 1, 0, 0\\
             0, 0, 0, 0, 0, 0, 0, 1, 0, 0, 0, 0, 0, 0, 0, 0, 1, 0\\              0, 0, 0, 0, 0, 0, 0, 0, 1, 0, 0, 0, 0, 1, 0, 0, 0, 0\\
               0, 0, 0, 0, 0, 0, 1, 0, 0, 0, 0, 0, 0, 0, 1, 0, 0, 0\\                0, 0, 0, 0, 0, 0, 1, 0, 0, 0, 0, 0, 1, 0, 0, 0, 0,
               0\end{pmatrix}.$$}

\end{ex}

\bibliographystyle{siam}
  \end{document}